\documentclass[12pt]{article}
\usepackage{amsmath,amsfonts,epsfig}
\title{\bf JSJ Decompositions of Coxeter Groups over FA Subgroups}
\author{John Ratcliffe and Steven Tschantz \\ 
Mathematics Department, Vanderbilt University, \\
Nashville TN 37240, USA}
\newtheorem{theorem}{Theorem}[section]

\newtheorem{lemma}[theorem]{Lemma}
\newtheorem{corollary}[theorem]{Corollary}
\newtheorem{conjecture}[theorem]{Conjecture}
\newenvironment{example}{\vspace{.2in}{\noindent\bf Example:\ }}{}
\newenvironment{remark}{\vspace{.2in}{\noindent\bf Remark: }}{}
\newenvironment{proof}{{\bf Proof:\ }}{\hfill$\square$\vspace{.2in}}

\date{}%May 26, 2009}
\begin{document}
\maketitle

\noindent {\bf Abstract:} A group $G$ has property FA if $G$ fixes a point of every tree 
on which $G$ acts without inversions.  
We prove that every Coxeter system of finite rank has a visual JSJ decomposition over subgroups with property FA.  As an application, we reduce the twist conjecture to Coxeter systems that are indecomposable with respect to amalgamated products over visual subgroups with property FA. 

\section{Introduction} % 1

JSJ decompositions first appeared in 3-manifold theory as secondary decompositions 
of 3-manifolds over tori.  
The primary decompositions of 3-manifolds being decompositions over 2-spheres. 
There is a close relationship between the topology of a 3-manifold and the algebraic properties of its fundamental group, so it was natural for JSJ decompositions to migrate to group theory.  

A JSJ decomposition of a group $G$, over a class of subgroups ${\cal A}$, is a graph of groups decomposition of $G$, with edge groups in ${\cal A}$, that has certain universal properties.  
JSJ decompositions of groups over various classes of subgroups have been worked out by a number of authors.  
For an introduction to JSJ decompositions of groups, see Guirardel and Levitt \cite{G-L}. 
Mihalik  \cite{M} recently worked out nice JSJ decompositions of Coxeter groups  over virtually abelian subgroups. 
In this paper, we work out nice JSJ decompositions of Coxeter groups over subgroups with property FA. 
Our JSJ decompositions of Coxeter groups are more primary than Mihalik's JSJ decompositions, and are analogous to the primary connect sum decompositions of 3-manifolds, whereas Mihalik's JSJ decompositions of Coxeter groups are analogous to the JSJ decompositions of 3-manifolds.

Let  $(W,S)$ be a Coxeter system of finite rank.  
By Lemma A of Bass \cite{Bass}, the graph of any graph of groups decomposition of $W$ is a tree, 
since the abelianization of $W$ is finite.  
Therefore a graph of groups decomposition $\Psi$ of $W$ is {\it reduced} if and only if 
no edge group of $\Psi$ is equal to a vertex group of $\Psi$. 
A {\it visual subgroup} of $(W,S)$ is a subgroup of $W$ generated by a subset of $S$. 
A {\it visual graph of groups decomposition} of $(W,S)$ is a graph of groups decomposition $\Psi$ of $W$ 
such that all the vertex and edge groups of $\Psi$ are visual subgroups of $(W,S)$.

A group $G$ has {\it property} FA if $G$ fixes a point of every tree on which $G$ acts without inversions.  
As a references for groups with property FA, see \cite{Bass} and \S I.6 of Serre's book \cite{Serre}. 
A subset $C$ of $S$ is said to be {\it complete} if the product of any two elements of $C$ has finite order.  
Note that the empty set is complete. 
If $C$ is a complete subset of $S$, we call $\langle C\rangle$ a complete visual subgroup of $W$.  
 A visual subgroup $\langle C\rangle$ of $(W,S)$ has property FA  if and only if $C$ is complete. 
 Mihalik and Tschantz  \cite{M-T} proved that if a subgroup $H$ of $W$ 
has property FA, then $H$ is contained in a conjugate of a complete visual subgroup of $(W,S)$ 
and the maximal subgroups of $W$ with property FA are the conjugates of the maximal 
complete visual subgroups of $(W,S)$. 

Let ${\cal FA}$ be the set of all subgroups of $W$ that are contained in some subgroup of $W$ 
with property FA.  Then ${\cal FA}$ is closed with respect to subgroups and conjugation. 
In this paper we prove that $(W,S)$ has a visual reduced JSJ decomposition $\Psi$ over the class of 
subgroups ${\cal FA}$.  All the edge groups of $\Psi$ are complete visual subgroups, 
and so have property FA.  The vertex groups of $\Psi$ are the maximal visual subgroups 
of $(W,S)$ that are indecomposable as an amalgamated product over a complete visual subgroup.  

We prove that sets of conjugacy classes of the vertex groups and the edge groups of a visual reduced JSJ decomposition $\Psi$ of $(W,S)$ over ${\cal FA}$ do not depend 
on the choice of the set of Coxeter generators $S$ of $W$. 

We prove that all the vertex groups of a visual reduced JSJ decomposition of $(W,S)$ over ${\cal FA}$ are complete if and only if $(W,S)$ is a chordal Coxeter system \cite{R-T}. 

As an application to the isomorphism problem for Coxeter groups, 
we reduce M\"uhlherr's twist conjecture \cite{Muhlherr} 
to Coxeter systems that are indecomposable as an amalgamated product over 
a visual complete subgroup.

\section{Existence of visual JSJ decompositions} % 2
 
We will use presentation diagrams to graphically represent 
Coxeter systems rather than Coxeter diagrams.  
The {\it presentation diagram} ({\it {\rm P}-diagram}) of 
a Coxeter system $(W,S)$ is the labeled undirected graph 
$\Gamma(W,S)$ with vertices $S$ and edges 
$\{(s,t) : s, t \in S\ \hbox{and}\ 1<\ m(s,t) < \infty\}$
such that an edge $(s,t)$ is labeled by the {\it order} $m(s,t)$ of $st$ in $W$. 
Note that a subset $C$ of $S$ is complete if and only if 
the underlying graph of $\Gamma(\langle C\rangle, C)$ is complete. 

Let $(W,S)$ be a Coxeter system of finite rank.  
Suppose that $S_1,S_2\subseteq S$, with $S=S_1\cup S_2$ and $S_0=S_1\cap S_2$,  
are such that $m(a,b)=\infty$ for all $a\in S_1-S_0$ and $b\in S_2-S_0$. 
Then we can write $W$ as a visual amalgamated product
$$W=\langle S_1\rangle*_{\langle S_0\rangle}\langle S_2\rangle.$$
 
We say that $S_0$ {\it separates} $S$ if
$S_1-S_0\neq\emptyset$ and $S_2-S_0\neq\emptyset$.  
The amalgamated product decomposition of $W$ will be
nontrivial if and only if $S_0$ separates $S$. 
If $S_0$ separates $S$, we call the triple $(S_1,S_0,S_2)$ 
a {\it separation} of $S$, and $S_0$ a {\it separator} of $S$. 
Note that $S_0$ separates $S$ if and only if $S_0$ 
separates $\Gamma(W,S)$, that is, there are $a,b$ in $S-S_0$ 
such that every path in $\Gamma(W,S)$ from $a$ to $b$ must pass through $S_0$. 
A subset $S_0$ of $S$ is a {\it minimal separator} of $S$, 
if $S_0$ separators $S$, and no other subset of $S_0$ separates $S$. 

\begin{lemma}  % 2.1
Let $(W,S)$ be a Coxeter system, and let $(S_1,S_0,S_2)$ be a separation of $S$ 
such that $S_0$ is complete.  If $T \subset S_1$ separates $S_1$, then $T$ separates $S$. 
\end{lemma}
\begin{proof}
On the contrary, suppose that $T$ does not separate $S$. 
As $T$ separates $S_1$, there exist $x, y$ in $S_1-T$ such that 
every path in $\Gamma(\langle S_1\rangle, S_1)$ from $x$ to $y$ passes through $T$. 
As $T$ does not separate $S$, there is a path in $\Gamma(W,S)$ from $x$ to $y$ 
that avoids $T$.  The path must exit $\Gamma(\langle S_1\rangle, S_1)$ through $S_0$ 
at some first element $a$ of $S_0$ before entering $S-S_1$ and must pass back through $S_0$ 
at some last element $b$ of $S_0$.  As $S_0$ is complete, we can short circuit 
the path by going directly from $a$ to $b$. 
This gives a path from $x$ to $y$ in $\Gamma(\langle S_1\rangle, S_1)$ that avoids $T$,  
which is a contradiction.  Thus $T$ must separate $S$. 
\end{proof}

\begin{lemma} % 2.2 
Let $(W,S)$ be a Coxeter system of finite rank. 
Then $(W,S)$ has a visual reduced graph of groups decomposition $\Psi$ 
such that for each vertex system $(V,R)$ of $\Psi$, the set $R$ is not separated by a complete subset, 
and such that each edge group of $\Psi$ is a complete visual subgroup of $(W,S)$. 
\end{lemma}
\begin{proof}
The proof is by induction on $|S|$. 
Suppose that $S$ is not separated by a complete subset. 
This includes the case $|S| = 1$. 
Let $\Psi$ be the trivial graph of groups decomposition of $W$ with one vertex and no edges. 
Then $\Psi$ satisfies the requirements of the lemma. 
Suppose  the lemma is true for all Coxeter systems of rank less than $|S|$ 
and $S$ is separated by a complete subset $S_0$. 
As every subset of $S_0$ is complete, we may assume that $S_0$ 
is a minimal separator of $S$. 
Let $(S_1,S_0,S_2)$ be a separation of $S$. 
Then $|S_i| < |S|$ for each $i=1,2$. 
By the induction hypothesis, $(\langle S_i\rangle, S_i)$ has a visual 
reduced graph of groups decomposition $\Psi_i$ satisfying the requirements of the lemma 
for each $i = 1,2$. 
As $S_0$ is complete, $S_0$ is not separated by any subset. 
Hence $S_0$ is contained in a vertex group $V_i$ of $\Psi_i$ for each $i =1,2$. 
We define a visual graph of groups decomposition of $(W,S)$ 
whose graph is obtained by joining the graph of $\Psi_1$ to the graph of $\Psi_2$ by 
an edge from the vertex of the graph of $\Psi_1$ corresponding to $V_1$ 
to the vertex of the graph of $\Psi_2$ corresponding to $V_2$. 
The vertex groups of $\Psi$ are the vertex groups of $\Psi_1$ and $\Psi_2$ assigned to their previous vertices.  The edge groups of $\Psi$ are the edge groups of $\Psi_1$ and $\Psi_2$, assigned to their previous edges, together with the group $\langle S_0\rangle$ assigned to the new edge. 

We next show that $\Psi$ is reduced. 
First assume $V_i = \langle S_i\rangle$ for some $i = 1,2$. 
Then $\langle S_0\rangle \neq V_i$, since $S_0 \neq S_i$. 
Now assume $V_i \neq \langle S_i\rangle$. 
Then there is an edge group $E$ of $\Psi_1$ incident to $V_1$. 
Let $T \subset S_1$ be the set of visual generators of $E$. 
Then $T$ separates $S_1$, and so $T$ separates $S$ by Lemma 2.1. 
Now $\langle S_0\rangle \neq V_i$, since otherwise $S_0$ would contain $T$ 
as a proper subset contradicting the fact that $S_0$ is a minimal separating subset of $S$. 
Hence $\Psi$ is reduced. 
Thus $\Psi$ has all the required properties. 
This completes the induction. 
\end{proof}

\begin{lemma}  % 2.3 
Let $(W,S)$ be a Coxeter system of finite rank, 
let $\Psi$  be a visual reduced graph of groups decomposition of $(W,S)$ 
such that for each vertex system $(V,R)$ of $\Psi$, the set $R$ is not separated 
by a complete subset, and such that 
each edge group of $\Psi$ is  a complete visual subgroup of $(W,S)$,   
and let $\Phi$ be a graph of groups decomposition of $W$ with edge groups in  ${\cal FA}$. 
Then each vertex group of $\Psi$ is contained in a conjugate of a vertex group of $\Phi$. 
\end{lemma}
\begin{proof}
Let $(V,R)$ be a vertex system of $\Psi$. 
By Theorem 1 of \cite{M-T}, the Coxeter system $(V,R)$ has a visual graph of groups decomposition 
$\Lambda$ such that each vertex group of $\Lambda$ is contained in a conjugate of a vertex group of $\Phi$ and each edge group of $\Lambda$ is contained in a conjugate of an edge group of $\Phi$.  
As $R$ is finite, we may assume that the graph of $\Lambda$ has only finitely many vertices and edges, 
and that $\Lambda$ is reduced. 
Let $(E,T)$ be an edge system of $\Lambda$. 
Then $E \in {\cal FA}$, since the edge groups of $\Phi$ are in ${\cal FA}$.  
Hence there is a complete subset $C$ of $S$ such that $E$ in contained in a conjugate of $\langle C\rangle$ 
by Lemma 25 of \cite{M-T}.  
This implies that $T$ is conjugate to a subset of $C$ by Lemma 4.3 of \cite{M-R-T}. 
Therefore $T$ is complete. 
Hence $R$ is separated by a complete subset, which is a contradiction. 
Therefore the graph of $\Lambda$ consists of a single point, 
and so $V$ is contained in a conjugate of a vertex group of $\Phi$. 
\end{proof}

The next theorem together with Lemma 2.2 imply that visual reduced JSJ decompositions 
over ${\cal FA}$ of a Coxeter system of a finite rank exist. 

\begin{theorem}  %  2.4
Let $(W,S)$ be a Coxeter system of finite rank, 
and let $\Psi$  be a visual reduced graph of groups decomposition of $(W,S)$ 
such that for each vertex system $(V,R)$ of $\Psi$, the set $R$ is not separated 
by a complete subset, and such that 
each edge group of $\Psi$ is  a complete visual subgroup of $(W,S)$.  
Then $\Psi$ is a JSJ decomposition of $W$ over the class  ${\cal FA}$. 
\end{theorem}
\begin{proof}
According to Guirardel and Levitt \cite{G-L}, we need to show that $\Psi$ is 
minimal, universally elliptic, and that $\Psi$ dominates every minimal universally elliptic graph of groups decomposition of $W$ over ${\cal FA}$. 
For a discussion of minimal graph of groups decompositions, see \S 2 of \cite{Forester}. 
The graph of groups decomposition $\Psi$ is minimal, since it is reduced, 
and universally elliptic, since the edge groups of $\Psi$ have property FA. 
By Lemma 2.3, the graph of groups decomposition $\Psi$ dominates every 
minimal graph of groups decomposition of $W$ over ${\cal FA}$. 
Thus $\Psi$ is a JSJ decomposition of $W$ over ${\cal FA}$. 
\end{proof}

Let $(W,S)$ be a Coxeter system of finite rank. 
Let $S_0\subset S$, and let $a,b\in  S-S_0$. 
We say that $S_0$ is an {\it $(a,b)$-separator} of $S$ if there is a separation $(S_1,S_0,S_2)$ of $S$ 
such that $a\in S_1-S_0$ and $b\in S_2-S_0$. 
Note that $S_0$ is an $(a,b)$-separator  of $S$ if and only if $a$ and $b$ lie in different connected components of $\Gamma(\langle S-S_0\rangle, S-S_0)$; 
moreover, $S_0$ separates $S$ if and only if there are elements $a, b\in S-S_0$ 
such that $S_0$ is an $(a,b)$-separator of $S$. 
We say that $S_0$ is a {\it minimal $(a,b)$-separator} of $S$ if $S_0$ is an $(a,b)$-separator of $S$ 
and no other subset of $S_0$ is an $(a,b)$-separator of $S$. 
We say that $S_0$ is a {\it relative minimal separator} of $S$ if there exists elements $a,b\in  S$ 
such that $S_0$ is a minimal $(a,b)$-separator of $S$. 
Note that every minimal separator of $S$ is a relative minimal separator of $S$, 
but a relative minimal separator of $S$ need not be a minimal separator of $S$. 

The next theorem characterizes the vertex groups and the edge groups of our JSJ decompositions. 

\begin{theorem} % 2.5
Let $(W,S)$ be a Coxeter system of finite rank, 
and let $\Psi$  be a visual reduced graph of groups decomposition of $(W,S)$ 
such that for each vertex system $(V,R)$ of $\Psi$, the set $R$ is not separated 
by a complete subset, and such that 
each edge group of $\Psi$ is  a complete visual subgroup of $(W,S)$.  
Let ${\cal V}$ be the set of all maximal subsets of $S$ that are not separated by a complete subset, 
and let ${\cal E}$ be the set of all complete, relative, minimal, separators of $S$. 
Then all the subgroups generated by sets in ${\cal V}$ are the vertex groups of $\Psi$, 
and all the subgroups generated by sets in ${\cal E}$ are the edge groups of $\Psi$. 
\end{theorem}
\begin{proof}
If $(E,T)$ is an edge system of $\Psi$, then $T$ is a separating subset of $S$, 
since the graph of $\Psi$ is a tree. 
Let $(V,R)$ be a vertex system of $\Psi$. 
Clearly, every subset of $S$ that contains $R$ properly  
is separated by a complete subset $C$ of $S$ that is contained in some edge group of $\Psi$ 
that is incident to $V$. 
Therefore $R$ is a maximal subset of $S$ 
that is not separated by a complete subset, 
and so $R\in {\mathcal V}$.

Now suppose $R\in{\mathcal V}$. 
We claim that $\langle R\rangle$ is a vertex group of $\Psi$. 
Every element of $R$ is in some vertex group of $\Psi$.
Let $R'\subseteq R$ be a maximal subset of $R$ that is contained in
some vertex group of $\Psi$. If $R-R'\neq\emptyset$, 
say $x\in R-R'$, then $R'$ and $x$ are not both contained in a vertex group
of $\Psi$.  Take vertex groups $V$ and $V'$ of $\Psi$,
with $x\in V$ and $R'\subseteq V'$, which are closest together
in the graph of $\Psi$.  Let $E$ be an edge group of the path
between $V$ and $V'$. Then $E$ is generated by a complete subset $T$ of $S$ 
by assumption.  Let $C = R\cap T$.  Then $C$ is a complete subset of $S$. 
Now $x\notin C$ otherwise $x$ would also
be in a vertex group closer to $V'$ on the path between $V$ and $V'$. 
Likewise, $R'\not\subseteq C$ or else $R'$ would be contained
in a vertex group closer to $V$ on a path between $V$ and $V'$. 
But then $\Gamma(\langle R-C\rangle,R-C)$ would have at least two connected components, 
one containing $x$ and one containing some element of $R'-C$.  
This contradicts the assumption that $R\in{\mathcal V}$.  
Instead all of $R$ must be contained in a vertex group $V$ of $\Psi$. 
By the maximality of $R$, we have that $\langle R\rangle = V$. 

Let $(E,T)$ be an edge system of $\Psi$. 
As the graph of $\Psi$ is a tree, there are distinct vertex systems $(V_1,R_1)$ and $(V_2,R_2)$ 
such that $R_1\cap R_2 = T$.  As $\Psi$ is reduced, there is an $a\in R_1-T$ and a $b\in R_2-T$. 
As the graph of $\Psi$ is a tree, $T$ is an $(a,b)$-separator of $S$. 
Let $t\in T$, and let $T'$ be any subset of $T$ not containing $t$. Then $T'$ is complete.  
Hence $T'$ does not separate $R_1$ or $R_2$. 
Therefore there is a path in $\Gamma(\langle R_1-T'\rangle, R_1-T')$ from $a$ to $t$ 
and there is a path in $\Gamma(\langle R_2-T'\rangle, R_2-T')$ from $t$ to $b$. 
Thus $T'$ is not a $(a,b)$-separator of $S$. 
Hence $T$ is a minimal $(a,b)$-separator of $S$. 
Thus $T\in {\cal E}$. 

Finally, suppose $T\in{\mathcal E}$.  
Then $T$ is a minimal $(a,b)$-separator of $S$ for some $\{a,b\} \subseteq S-T$. 
Let $(S_1,T,S_2)$ be a separation of $S$ with $a\in S_1-T$ and $b\in S_2-T$.  
Each $R\in{\mathcal V}$ generates a vertex group of
$\Psi$ and is not separated by  any subset of $T$, 
and so each $R\in{\mathcal V}$
is contained in either $S_1$ or $S_2$.  

Pick vertex groups $V_1$
and $V_2$ as close together in $\Psi$ as possible such that
$V_1$ is generated by a subset of $S_1$ and $V_2$ is generated by
a subset of $S_2$.  Then $V_1$, and $V_2$ are adjacent since every
vertex group in a path between $V_1$ and $V_2$ is generated by a subset of
either $S_1$ or $S_2$.  Now $V_1\cap V_2$ is an edge group $E$ of
$\Psi$ which is generated by a subset $T'$ of $T$. 
The set $T'$ is an $(a,b)$-separator of $S$ since the graph of $\Psi$ is a tree. 
Hence $T' = T$, since $T$ is a minimal $(a,b)$-separator of $S$.  
Thus the sets in ${\cal E}$ generate the edge groups of $\Psi$. 
\end{proof}

$$\mbox{
\setlength{\unitlength}{.8cm}
\begin{picture}(12,6)(0,0)
\thicklines
\put(2,3){\circle*{.15}}
\put(2,3){\line(1,0){7}}
\put(5,3){\circle*{.15}}
\put(9,3){\circle*{.15}}
\put(5,3){\line(4,5){2}}
\put(5,3){\line(4,-5){2}}
\put(7,5.5){\circle*{.15}}
\put(7,.5){\circle*{.15}}
\put(7,5.5){\line(4,-5){2}}
\put(7,.5){\line(4,5){2}}
\put(1.4,2.9){$a$}
\put(4.7,2.3){$b$}
\put(7.1,5.8){$c$}
\put(7.1,0){$d$}
\put(9.3,2.9){$e$}
\end{picture}}$$
\centerline{\bf Figure 1.  A graph with five vertices $a, b, c, d, e$}

\begin{example}  Consider a Coxeter system $(W,S)$ such that the underlying graph 
of $\Gamma(W,S)$ is given in Figure 1. Then $(W,S)$ has two visual reduced JSJ 
decompositions over ${\cal FA}$, namely, 
$$W = \langle a,b\rangle\ast_{\langle b\rangle}\langle b,c,e\rangle\ast_{\langle b,e\rangle}\langle b,d,e\rangle, $$
$$W = \langle a,b\rangle\ast_{\langle b\rangle}\langle b,d,e\rangle\ast_{\langle b,e\rangle}\langle b,c,e\rangle. 
$$
By Theorem 2.5, both decompositions have the same vertex groups and the same edge groups. 
The only difference between the two decompositions is their graphs. 
In the first decomposition the edge group $\langle b\rangle$ is attached to the vertex group $\langle b,c,e\rangle$ whereas in the second decomposition, the edge group $\langle b\rangle$ is attached to the vertex group $\langle b,d,e\rangle$.  
The two decompositons are related by a slide move \cite{G-L} (Definition 7). 
It is worth noting that $\{b,e\}$ is a relative minimal separator of $S$, 
but $\{b,e\}$ is not a minimal separator of $S$, since $\{b\}$ separates $S$. 
\end{example}

\begin{remark}  Let $(W,S)$ be a Coxeter system of finite rank. 
The decomposition of the underlying graph of $\Gamma(W,S)$ determined by a visual reduced 
JSJ decomposition of $(W,S)$ over ${\cal FA}$ was first described in a graph theoretic context by Leimer \cite{Leimer}. 
In particular, an efficient algorithm for finding such a decomposition is given in Leimer's paper. 
It is interesting that such decompositions of graphs arise here naturally in the theory of Coxeter groups. 
\end{remark}

\section{Uniqueness of JSJ decompositions}  % 3

We now turn our attention to the uniqueness of reduced JSJ decompositions 
of a Coxeter group $W$. 
Let ${\cal A}$ be a class of subgroups of $W$ which is closed with respect to taking subgroups 
and conjugation. 

\begin{theorem}  % 3.1
Let $(W,S)$ be a Coxeter system of finite rank, 
and let $\Psi$ and $\Psi'$ be reduced JSJ decompositions of $W$ over ${\cal A}$. 
Then for each vertex group $V$ of $\Psi$, there is a unique vertex group $V'$ of $\Psi'$ 
such that $V$ is conjugate to $V'$ in $W$. 
Therefore the graphs of $\Psi$ and $\Psi'$ have the same number of vertices.
\end{theorem}
\begin{proof}
Let $V$ be a vertex group of $\Psi$.  Then there is a $w\in W$ and a vertex group $V'$ of $\Psi'$ 
such that $V \subseteq wV'w^{-1}$, since $\Psi$ dominates $\Psi'$ by Theorem 12 of \cite{G-L}. 
Moreover, there is a $w'\in W$ and a vertex group $V''$ of $\Psi$ such that $V' \subseteq w'V''w'^{-1}$, 
since $\Psi'$ dominates $\Psi$ by Theorem 12 of \cite{G-L}.  
Hence $V \subseteq ww'V''(ww')^{-1}$.  
By Lemma 3 of \cite{M-T}, we have that $V = V''$ and $ww'\in V$. 
As $V \subseteq wV'w^{-1} \subseteq ww'V(ww')^{-1} = V$, 
we have that $V = wV'w^{-1}$; moreover $V'$ is unique by Lemma 3 of \cite{M-T}. 
\end{proof}

\begin{lemma} % 3.2
Let $(W,S)$ be a Coxeter system of finite rank. 
Then $S$ is separated by a complete subset if and only if 
$W$ has a nontrivial amalgamated product decomposition over a subgroup in ${\cal FA}$. 
\end{lemma}
\begin{proof}
Suppose $S_0$ is a complete separator of $S$. 
Then there is a separation $(S_1,S_0,S_2)$ of $S$ 
and we have a nontrivial amalgamated product decomposition
$W=\langle S_1\rangle*_{\langle S_0\rangle}\langle S_2\rangle$ with $\langle S_0\rangle \in {\cal FA}$. 

Conversely, suppose $W$ has a nontrivial amalgamated product decomposition
$W=A*_C B$ with $C \in {\cal FA}$, and on the contrary, $S$ has no complete separator. 
By Lemma 2.3, we have that $W$ is contained in a conjugate of $A$ or $B$,  
which is a contradiction. 
Therefore $S$ has a complete separator. 
\end{proof}

The next theorem together with Theorem 2.4 characterize a visual reduced JSJ decomposition over 
${\cal FA}$ of a Coxeter system $(W,S)$ of finite rank.  

\begin{theorem} % 3.3
Let $\Psi$ be a visual reduced JSJ decomposition of over ${\cal FA}$ of a Coxeter system $(W,S)$ of finite rank. Then  for each vertex system $(V,R)$ of $\Psi$, the set $R$ is not separated by a complete subset, and each edge group of $\Psi$ is a complete visual subgroup of  $(W,S)$. 
\end{theorem}
\begin{proof}
By Lemma 2.2, the system $(W,S)$ has a visual reduced graph of groups decomposition $\Psi'$ 
such that for each vertex system $(V',R')$ of $\Psi'$, the set $R'$ is not separated by a complete subset, 
and such that each edge group of $\Psi'$ is a complete visual subgroup of $(W,S)$. 
By Theorem 2.4, the decomposition $\Psi'$ is a JSJ decomposition of $W$ over the class ${\cal FA}$. 
Let $(V,R)$ be a vertex system of $\Psi$.  By Theorem 3.1, there is a vertex system $(V',R')$ of $\Psi'$ 
such that $V$ is conjugate to $V'$.  By Lemma 3.2, the group $V'$ is indecomposable as 
a nontrivial amalgamated product over a subgroup in ${\cal FA}$. 
Hence $V$ is indecomposable as a nontrivial amalgamated product over a subgroup in ${\cal FA}$. 
By Lemma 3.2, the set $R$ is not separated by a complete subset. 

Let $(E,T)$ be an edge system of $\Psi$.  Then $E\in {\cal FA}$. 
Hence $E$ is contained in a FA subgroup $H$ of $W$. 
By Lemma 25 of \cite{M-T}, there is a complete subset $C$ of $S$ and a $w\in W$ 
such that $H \subseteq w\langle C\rangle w^{-1}$. 
Hence $E \subseteq w\langle C\rangle w^{-1}$. 
By Lemma 4.3 of \cite{M-R-T}, the set $T$ is conjugate to a subset of $C$. 
Hence $T$ is complete and $E$ is a complete visual subgroup $(W,S)$. 
\end{proof}

Let $(W,S)$ be a Coxeter system of finite rank.  
A subset $S_0$ of $S$ is a {\it c-minimal separator} of $S$ 
if $S_0$ separates $S$ and no conjugate of another subset of $S_0$ separates $S$. 
Note that if $S_0$ is a c-minimal separator of $S$, then $S_0$ is a minimal separator of $S$. 
Also if $S_0$ and $S_0'$ are conjugate separators of $S$, then $S_0$ is c-minimal if and only 
if $S_0'$ is c-minimal. 

$$\mbox{
\setlength{\unitlength}{.8cm}
\begin{picture}(12,5)(0,0)
\thicklines
\put(2,3){\circle*{.15}}
\put(2,3){\line(1,0){3}}
\put(5,3){\circle*{.15}}
\put(9,3){\circle*{.15}}
\put(5,3){\line(4,3){2}}
\put(5,3){\line(4,-3){2}}
\put(7,4.5){\circle*{.15}}
\put(7,1.5){\circle*{.15}}
\put(7,1,5){\line(0,1){3}}
\put(7,4.5){\line(4,-3){2}}
\put(7,1.5){\line(4,3){2}}
\put(1.4,2.9){$a$}
\put(3.4,3.3){$3$}
\put(4.7,2.3){$b$}
\put(5.7,4){$3$}
\put(5.7,1.7){$2$}
\put(6.9,4.9){$c$}
\put(7.3,2.8){$3$}
\put(8.1,4){$2$}
\put(6.9,.9){$d$}
\put(8.1,1.7){$2$}
\put(9.3,2.9){$e$}
\end{picture}}$$
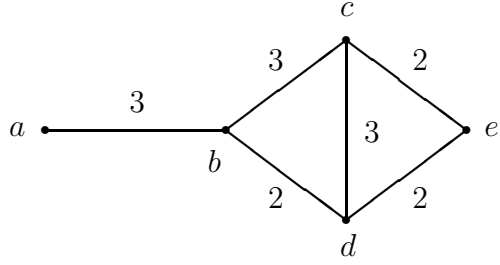
\centerline{\bf Figure 2.  The P-diagram of a Coxeter system}

\begin{example}  Consider the Coxeter system $(W,S)$ whose P-diagram  is given in Figure 2. 
Observe that $\{c,d\}$ is a minimal separator of $S$, but $\{c,d\}$ is not a c-minimal separator of $S$, 
since $c$ is conjugate to $b$ and $\{b\}$ separates $S$. 
\end{example}

\begin{lemma}  % 3.4
Let $(W,S)$ be a Coxeter system of finite rank, 
and let $S'$ be another set of Coxeter generators of $W$. 
If $S_0$ is a c-minimal separator of $S$, then there exists a c-minimal separator of $S'_0$ of $S'$ 
such that $\langle S_0\rangle$ is conjugate to $\langle S_0'\rangle$ in $W$. 
\end{lemma}
\begin{proof}
Let $(S_1,S_0,S_2)$ be a separation of $S$. 
Then we have a nontrivial visual amalgamated product decomposition
$W=\langle S_1\rangle*_{\langle S_0\rangle}\langle S_2\rangle$. 
By Theorem 6.1 of \cite{M-R-T}, there is a c-minimal separator $S_0''$ of $S$ 
and a c-minimal separator $S_0'$ of $S'$ such that $\langle S_0''\rangle$ 
is conjugate to $\langle S_0'\rangle$ and $\langle S_0''\rangle$ 
is conjugate to a subgroup of $\langle S_0\rangle$. 
By Lemma 4.3 of \cite{M-R-T}, we have that $S_0''$ is conjugate to a subset of $S_0$. 
As $S_0$ is a c-minimal separator of $S$, we deduce that $S_0''$ is conjugate to $S_0$. 
Hence $\langle S_0\rangle$ is conjugate to $\langle S_0'\rangle$ in $W$. 
\end{proof}

We now turn our attention to the uniqueness of the edge groups of a visual 
reduced JSJ decomposition over ${\cal FA}$ of a Coxeter system. 

\begin{theorem}  % 3.5
Let $(W,S)$ be a Coxeter system of finite rank, 
and let $\Psi$ be a visual reduced JSJ decomposition of $(W,S)$ over ${\cal FA}$. 
Let $S'$ be another set of Coxeter generators of $W$, 
and let $\Psi'$ be a visual reduced JSJ decomposition of $(W,S')$ over ${\cal FA}$. 
Then for each edge group $E$ of $\Psi$, 
there is an edge group $E'$ of $\Psi'$ such that $E$ is conjugate to $E'$ in $W$.  
Moreover for each edge system $(E,T)$ of $\Psi$ such that $T$ is a c-minimal separator of $S$, 
there is an edge system $(E',T')$ of $\Psi'$ such that $E$ is conjugate to $E'$ and $T'$ 
is a c-minimal separator of $S'$. 
\end{theorem}
\begin{proof} 
According to Tits \cite{Tits}, Bass \cite{Bass}, and Guirardel and Levitt \cite{G-L2}, 
there are five possible types of reduced JSJ decompositions of $W$ over ${\cal FA}$, 
namely, trivial, dihedral, linear abelian, genuine abelian, and irreducible.  
The abelian types do not apply to $W$, since the abelianization of $W$ is finite. 
By Proposition 3.10 of \cite{G-L2}, the decompositions $\Psi$ and $\Psi'$ have the same type. 
If $\Psi$ and $\Psi'$ are both trivial, then they have no edge groups. 

Suppose that $\Psi$ and $\Psi'$ are dihedral. 
By Theorem 6 of \S I.4 of \cite{Serre}, we deduce that 
$\Psi$ corresponds to a nontrivial visual amalgamated product decomposition 
$W = \langle A\rangle\ast_{\langle C\rangle} \langle B\rangle$ with $C$ complete 
and $\langle C\rangle$ of index two in both $\langle A\rangle$ and $\langle B\rangle$, 
and $\Psi'$ corresponds to a nontrivial visual amalgamated product decomposition 
$W = \langle A'\rangle\ast_{\langle C'\rangle} \langle B'\rangle$ with $C'$ complete 
and $\langle C'\rangle$ of index two in both $\langle A'\rangle$ and $\langle B'\rangle$. 
Now $\langle C\rangle$ and $\langle C'\rangle$ are normal in $W$. 
By the main result of \cite{Paris}, we deduce that $A-C = \{a\}$ and $\langle A\rangle = \langle a\rangle \times\langle C\rangle$, and $B-C = \{b\}$ and $\langle B\rangle = \langle b\rangle\times\langle C\rangle$, 
and $A'-C' = \{a'\}$ and $\langle A'\rangle = \langle a'\rangle \times\langle C'\rangle$, and $B'-C' = \{b'\}$ and $\langle B'\rangle = \langle b'\rangle\times\langle C'\rangle$. 
Hence $A, B, A', B'$ are all complete.  
Therefore $C$ is the unique separator of $S$ and $C'$ is the unique separator of $S'$. 
Hence $\langle C\rangle$ is conjugate to $\langle C'\rangle$ by Lemma 3.4. 

Now assume that $\Psi$ and $\Psi'$ are irreducible. 
Then $\Psi$ and $\Psi'$ are non-ascending \cite{G-L2}, since the graphs of $\Psi$ and $\Psi'$ are trees, 
and so for each edge group $E$ of $\Psi$, 
there is an edge group $E'$ of $\Psi'$ such that $E$ is conjugate to $E'$ in $W$ by 
Corollary 7.3 of \cite{G-L2}.   

Let $(E,T)$ be an edge system of $\Psi$ such that $T$ is a c-minimal separator of $S$. 
Then there is a c-minimal separator $T'$ of $S'$ such that $E$ is conjugate to $ \langle T'\rangle$ by Lemma 3.4.  As $E$ has property FA, we have that $\langle T'\rangle$ has property FA. 
Therefore $T'$ is complete.  Hence $T'$ generates an edge group $E'$ of $\Psi'$ by Theorem 2.5. 
\end{proof}

\vspace{-.2in}
\begin{remark}  Let $\Psi$ and $\Psi'$ be as in Theorem 3.5.  
It is not necessary that $\Psi$ and $\Psi'$ have the same number of edge groups. 
For M\"uhlherr's example \cite{M}, we have $\Psi$ with one edge group and $\Psi'$ with two edge groups. 
It is also not necessary that a minimal edge group of $\Psi$ is conjugate to a minimal edge group of $\Psi'$. 
We will give an example below. 
\end{remark}

\vspace{.2in}
Let $(W,S)$ be a Coxeter system of finite rank.  
Suppose that $S_1,S_2\subseteq S$, with $S=S_1\cup S_2$ and $S_0=S_1\cap S_2$,  
are such that $m(a,b)=\infty$ for all $a\in S_1-S_0$ and $b\in S_2-S_0$. 
Let $\ell \in \langle S_0\rangle$ such that 
$\ell S_0\ell^{-1}=S_0$. 
The triple $(S_1,\ell, S_2)$ determines an {\it elementary twist} 
of $(W,S)$ giving a new Coxeter generating set $S'=S_1\cup \ell S_2\ell^{-1}$ of $W$. 
Note that if $S_1 \subseteq S_2$, then $S'=\ell S\ell^{-1}$.

$$\mbox{
\setlength{\unitlength}{.8cm}
\begin{picture}(12,6.4)(0,-.2)
\thicklines
\put(2,3){\circle*{.15}}
\put(2,3){\line(1,0){7}}
\put(5,3){\circle*{.15}}
\put(9,3){\circle*{.15}}
\put(5,3){\line(4,5){2}}
\put(5,3){\line(4,-5){2}}
\put(7,5.5){\circle*{.15}}
\put(7,.5){\circle*{.15}}
\put(7,5.5){\line(4,-5){2}}
\put(7,.5){\line(4,5){2}}
\put(1.4,2.9){$a$}
\put(3.4,3.3){$3$}
\put(3,2.3){$b=\ell d\ell^{-1}$}
\put(5.7,4.5){$2$}
\put(5.7,1.2){$2$}
\put(7,3.3){$3$}
\put(6.6,5.9){$\ell e\ell^{-1}$}
\put(8.1,4.5){$2$}
\put(5.7,-.2){$d=\ell b\ell^{-1}$}
\put(8.1,1.2){$3$}
\put(9.3,2.9){$c=\ell c\ell^{-1}$}
\end{picture}}$$
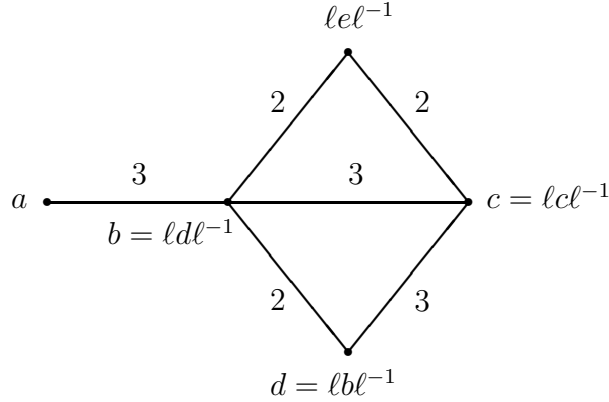
\centerline{\bf Figure 3.  The P-diagram of a twisted Coxeter system}

\begin{example}  Consider the Coxeter system $(W,S)$ whose P-diagram is given in Figure 2. 
Let $\ell$ be the longest element of the visual subgroup $\langle b,c,d\rangle$. 
Then $(\{a,b,c,d\},\ell, \{b,c,d,e\})$ is an elementary twist of $(W,S)$. 
The P-diagram of the twisted system $(W,S')$ is given in Figure 3. 
Let $\Psi$ be the unique visual reduced JSJ decomposition of $(W,S)$ over ${\cal FA}$, 
and let $\Psi'$ be one of the two visual reduced JSJ decompositions of $(W,S')$ over ${\cal FA}$. 
The minimal edge group $\langle c, d\rangle$ of $\Psi$ is conjugate to the edge group $\langle b, c\rangle$ 
of $\Psi'$.  The edge group $\langle b,c\rangle$ is not minimal, since $\langle b\rangle$ 
is an edge group of $\Psi'$. 
\end{example}

\section{Chordal Coxeter Groups}

A graph is said to be {\it chordal} if every cycle of length at least four has a chord. 
For example, the graph in Figure 1 is chordal. 
A Coxeter system $(W,S)$ is said to be {\it chordal} if the underlying graph of the P-diagram of $(W,S)$ 
is chordal.  

\begin{theorem}  % 4.1
Let $(W,S)$ be a Coxeter system of finite rank, 
and let $\Psi$ be a visual reduced JSJ decomposition of $(W,S)$ over ${\cal FA}$. 
Then $(W,S)$ is chordal if and only if each vertex group of $\Psi$ is a complete 
visual subgroup of $(W,S)$. 
\end{theorem}
\begin{proof}
Suppose $(W,S)$ is chordal. 
We prove that each vertex system of $\Psi$ is complete by induction on $|S|$. 
Suppose that $(W,S)$ is complete.  This includes the case $|S| =1$. 
Then $(W,S)$ is the only vertex system of $\Psi$ by Theorems 2.5 and 3.3, 
and $(W,S)$ is complete. 

Now suppose that $(W,S)$ is incomplete
and that each vertex system of a visual reduced JSJ decomposition over ${\cal FA}$ 
of a chordal Coxeter system, of rank less than $|S|$, is complete.  
Then $S$ has a separating subset. 
Let $S_0$ be a minimal separator of $S$. 
By Theorem 1 of \cite{Dirac}, the set $S_0$ is complete. 
Now $\langle S_0\rangle$ is an edge group of $\Psi$ by Theorems 2.5 and 3.3. 
As the graph of $\Psi$ is a tree, $\Psi$ is a nontrivial amalgamated 
product of two visual reduced graph of group decompositions $\Psi_1$ and $\Psi_2$ 
amalgamated along $\langle S_0\rangle$. 
By Theorems 2.4 and Theorem 3.3, we deduce that $\Psi_i$ 
is a visual reduced JSJ decomposition over ${\cal FA}$ of a proper subsystem $(W_i,S_i)$ of $(W,S)$ for $i=1,2$. 
Each subsystem of $(W,S)$ is chordal. 
By the induction hypothesis, each vertex system of $\Psi_1$ and $\Psi_2$ is complete. 
Therefore each vertex system of $\Psi$ is complete. 
This completes the induction. 

Conversely, suppose that each vertex system of $\Psi$ is complete. 
We prove that $(W,S)$ is chordal by induction on the number of vertices in the graph of $\Psi$. 
Suppose that the graph of $\Psi$ has only one vertex. 
Then the graph of $\Psi$ is a point, since the graph is a tree. 
Hence $(W,S)$ is complete.  Therefore $(W,S)$ is chordal. 

Now suppose that the graph of $\Psi$ has more than one vertex, 
and all finite rank Coxeter systems, with a visual reduced JSJ decomposition over ${\cal FA}$ 
with fewer vertices than $\Psi$ and all vertex systems complete, are chordal. 
Let $(E,T)$ be an edge system of $\Psi$. 
As the graph of $\Psi$ is a tree,  $\Psi$ is a nontrivial amalgamated 
product of two visual reduced graph of group decompositions $\Psi_1$ and $\Psi_2$ 
amalgamated along $\langle S_0\rangle$.  
By Theorems 2.4 and Theorem 3.3, we deduce that $\Psi_i$ 
is a visual reduced JSJ decomposition over ${\cal FA}$ of a proper subsystem $(W_i,S_i)$ of $(W,S)$ for  $i=1,2$. 
By the induction hypothesis, $(W_i,S_i)$ is chordal for $i=1,2$. 
Now $\Gamma(W,S) = \Gamma(W_1,S_1) \cup \Gamma(W_2,S_2)$ 
and $\Gamma(W_1,S_1)\cap \Gamma(W_2,S_2) = \Gamma(E,T)$ with $\Gamma(E,T)$ complete. 
Therefore $(W,S)$ is chordal by Theorem 2 of \cite{Dirac}. 
This completes the induction. 
\end{proof}

\section{Application to the Isomorphism Problem}  % 5

Let $(W,S)$ be a Coxeter system of finite rank. 
Define 
$$S^W = \{wsw^{-1}: s\in S\ \hbox{and}\ w\in W\}.$$
Let $S'$ be another set of Coxeter generators of $W$.  
The set of generators $S$ is said to be {\it sharp-angled} with respect to $S'$ if for each pair 
$s,t\in S$ such that $2<m(s,t)<\infty$, 
there is a $w\in W$ such that $w\{s,t\}w^{-1} \subseteq S'$. 
The Coxeter systems $(W,S)$ and $(W,S')$ are said to be {\it twist equivalent} if there is a finite sequence 
of elementary twists that transforms $S$ into $S'$. 
If $(W,S)$ and $(W,S')$ are twist equivalent, then  $S' \subseteq S^W$ 
and $S$ is sharp-angled with respect to $S'$.

The following conjecture is due to M\"uhlherr  \cite{Muhlherr} (Conjecture 2). 

\begin{conjecture}  % 5.1
{\rm (Twist Conjecture)} Let $(W,S)$ be a Coxeter system of finite rank, and let $S'$ 
be another set of  Coxeter generators of $W$ such that $S' \subseteq S^W$ and $S$ is sharp-angled with respect to $S'$. 
Then $(W,S)$ is twist equivalent to $(W,S')$. 
\end{conjecture}

\begin{lemma}  % 5.2 
If $(W,S)$ is a complete Coxeter system of finite rank,  
then $(W,S)$ satisfies the twist conjecture. 
\end{lemma}
\begin{proof}
Let $S'$ be another set of Coxeter generators for $W$ such that $S'\subseteq S^W$ 
and $S$ is sharp-angled with respect to $S'$. 
We need to prove that $(W,S)$ is twist equivalent to $(W,S')$. 
As $S$ has no separating subsets, $(W,S)$ can only be twisted by conjugating $S$. 
Now $W$ has property FA, since $S$ is complete.   
Hence $S'$ is complete by Lemma 25 of \cite{M-T}. 
Let $(W,S) = (W_1,S_1)\times\cdots\times(W_n,S_n)$
be the factorization of $(W,S)$ into irreducible factors, and let 
$(W,S') = (W_1',S_1')\times\cdots\times(W_m',S_m')$ 
be the factorization of $(W,S')$ into irreducible factors. 
By Lemma 14 of Franzsen and Howlett \cite{F-H}, 
$m=n$ and by reindexing, we may assume that $W_i'=W_i$ for each $i=1,\ldots,n$. 
As $S$ is sharp-angled with respect to $S'$, there is a $w_i\in W$ such that 
$w_iS_iw_i^{-1}\subseteq S'$ for each $i$ by Lemma 7.1 of \cite{R-T}. 
As the $j$th component of $w_i$, for $j\neq i$, centralizes $W_i$, 
we may assume that $w_i\in W_i$. 
Then $w_iS_iw_i^{-1} = S_i'$ for each $i$. 
Let $w=w_1\cdots w_n$. 
Then $wSw^{-1} = S'$. 
Therefore $(W,S)$ is twist equivalent to $(W,S')$. 
Thus $(W,S)$ satisfies the twist conjecture. 
\end{proof}

\begin{theorem}  % 5.3
Let $(W,S)$ be a Coxeter system of finite rank, and let $\Psi$ be a visual reduced JSJ decomposition of $(W,S)$ over ${\cal FA}$.  If each vertex system $(V,R)$ of $\Psi$ satisfies the twist conjecture, then $(W,S)$ satisfies the twist conjecture. 
\end{theorem}
\begin{proof}
Let $S'$ be another set of Coxeter generators for $W$ such that $S'\subseteq S^W$ 
and $S$ is sharp-angled with respect to $S'$. 
We need to prove that $(W,S)$ is twist equivalent to $(W,S')$. 
The proof is by induction on the number of vertices of the graph of $\Psi$. 
Suppose the graph of $\Psi$ has only one vertex.  Then the graph of $\Psi$ 
is a single point, since the graph of $\Psi$ is a tree. 
Hence $(W,S)$ satisfies the twist conjecture by hypothesis. 
Therefore $(W,S)$ is twist equivalent to $(W,S')$. 

Now assume that the graph of $\Psi$ has more than one vertex, 
and the theorem is true for all Coxeter systems of finite rank 
whose JSJ decompositions over ${\cal FA}$ have fewer vertex systems than $\Psi$. 
Then $S$ has a complete separating subset $C$. 
We now follow the argument in the proof of Theorem 8.4 of \cite{R-T}. 

Let $(A,C,B)$ be a separation of $S$. 
Then $W = \langle A\rangle \ast_{\langle C\rangle} \langle B\rangle$ 
is a nontrivial amalgamated product decomposition. 
By Theorem 6.6 of \cite{M-R-T}, the Coxeter systems 
$(W,S)$ and $(W,S')$ are twist equivalent 
to Coxeter systems $(W,R)$ and $(W,R')$, respectively, 
such that there exists a nontrivial
visual reduced graph of groups decomposition $\Phi$ of $(W,R)$
and a nontrivial visual visual graph of groups decomposition
$\Phi'$ of $(W,R')$ having the same graphs and the same vertex
and edge groups and all edge groups equal and a subgroup of a
conjugate of $\langle C \rangle$.
Now $R'\subseteq R^W$ and $R$ is sharp-angled with respect to $R'$, 
since $R'$ is twist equivalent to $S'$. 

Let $\{(W_i,R_i)\}_{i=1}^k$ be the Coxeter systems of the vertex groups of $\Psi$, 
and let $(W_0,R_0)$ be the Coxeter system of the edge group of $\Psi$. 
Then $k\geq 2$, and $R =\cup_{i=1}^k R_i$, and $\cap_{i=1}^k R_i = R_0$, 
and $R_i - R_0 \neq \emptyset$ for each $i > 0$, 
and $m(a,b) = \infty$ for each $a \in R_i-R_0$ and $b \in R_j-R_0$ with $i\neq j$. 
By Lemma 4.3 and Theorems 6.1 and 6.6 of \cite{M-R-T}, 
we have that $R_0$ is conjugate to a subset of $C$, and so $R_0$ is complete. 
By Theorems 6.1 and 6.6  of \cite{M-R-T}, we have that $R_0$ 
is conjugate to a c-minimal separator of $S$.  
By Lemma 3.4, there is a c-minimal separator $R_0''$ of $R$ such that $\langle R_0\rangle$ is conjugate 
to $\langle R_0''\rangle$.  By Lemma 4.3 of \cite{M-R-T}, we have that $R_0$ is conjugate to $R_0''$. 
As $R_0$ separates $R$, we conclude that $R_0$ is a c-minimal separator of $R$. 

Let $\Phi_i$ be a visual reduced JSJ decomposition of $(W_i,R_i)$ over ${\cal FA}$ 
for each $i=1,\ldots, k$. 
As $R_0$ is a complete minimal separator of $R$, we can amalgamate 
$\Phi_1,\ldots, \Phi_k$ to give a visual reduced JSJ decomposition $\Phi$ of $(W,R)$ over ${\cal FA}$ 
with the same vertex groups and the edge group $\langle R_0\rangle$ joining 
a vertex group of $\Psi_i$ to a vertex group of $\Psi_{i+1}$, for each $i=1,\ldots,k-1$,  
by the same argument as in the proof of Lemma 2.2.  
Hence the number of vertices in the graph of $\Phi_i$ is less than the number of vertices 
of the graph of $\Phi$ for each $i=1,\ldots, k$. 
By Theorem 3.1, the graphs of $\Phi$ and $\Psi$ have the same number of vertices.

Let $\{(W_i',R_i')\}_{i=1}^k$ be the Coxeter systems of the vertex groups of $\Psi'$ 
indexed so that $W_i' = W_i$ for each $i$, 
and let $(W_0',R_0')$ be the Coxeter system of the edge group of $\Psi'$. 
Then $W_0' = W_0$, and $R' =\cup_{i=1}^k R_i'$, and $\cap_{i=1}^k R_i' = R_0'$, 
and $R_i' - R_0' \neq \emptyset$ for each $i > 0$, 
and $m(a',b') = \infty$ for each $a' \in R_i'-R_0'$ and $b' \in R_j'-R_0'$ with $i\neq j$.

By Lemma 8.1 of \cite{R-T}, we have  $R_i'\subseteq R_i^{W_i}$ and  
$R_i$ is sharp-angled with respect to $R_i'$ for each $i$. 
Hence by the induction hypothesis, $(W_i,R_i)$ is twist equivalent to $(W_i,R_i')$ for each $i$. 
As $R_0$ is complete, there is an element $w_0$ of $W_0$ such that $w_0R_0w_0^{-1} = R_0'$ 
by the proof of Lemma 5.2.  
By conjugating $W$ by $w_0$, we may assume that $R_0=R_0'$. 
By the same argument as in the last paragraph of the proof of Theorem 8.4 of \cite{R-T}, 
we have that $(W,R)$ is twist equivalent to $(W,R')$, 
and so $(W,S)$ is twist equivalent to $(W,S')$. 
This completes the induction. 
\end{proof}

\begin{corollary}  % 5.4 
{\rm (Theorem 8.4, \cite{R-T})}
All Chordal Coxeter systems of finite rank satisfy the twist conjecture. 
\end{corollary}
\begin{proof} This follows from Theorem 4.1, Lemma 5.2, and Theorem 5.3. 
\end{proof}

\end{document}